\documentclass{elsarticle}
\usepackage[latin9]{inputenc}
\usepackage{a4}
\usepackage{amsfonts}
\usepackage{amssymb}
\usepackage{amsmath}
\usepackage{amsthm}
\usepackage{changepage}
\usepackage{nomencl}
\usepackage{geometry}
\usepackage{breqn}
\usepackage[toc,page]{appendix}
\begin{document}
\providecommand{\keywords}[1]{\textbf{\textit{Keywords: }} #1}
\newtheorem{satz}{Theorem}
\newtheorem{lemma}[satz]{Lemma}
\newtheorem{prop}[satz]{Proposition}
\newtheorem{kor}[satz]{Corollary}
\theoremstyle{definition}
\newtheorem{defi}{Definition}

\newcommand{\cc}{{\mathbb{C}}}   
\newcommand{\ff}{{\mathbb{F}}}  
\newcommand{\nn}{{\mathbb{N}}}   
\newcommand{\qq}{{\mathbb{Q}}}  
\newcommand{\rr}{{\mathbb{R}}}   
\newcommand{\zz}{{\mathbb{Z}}}  
\author{Joachim K\"onig} 
\title{Computation of Hurwitz spaces and new explicit polynomials for almost simple Galois groups} 
\address{Universit\"at W\"urzburg, Emil-Fischer-Str.\ 30, 97074 W\"urzburg, Germany}
\ead{joachim.koenig@mathematik.uni-wuerzburg.de}
\date{}
{\abstract{We compute the first explicit polynomials with Galois groups $G=P\Gamma L_3(4)$, $PGL_3(4)$, $PSL_3(4)$ and $PSL_5(2)$ over $\qq(t)$.
Furthermore we compute the first examples of totally real polynomials with Galois groups $PGL_2(11)$, $PSL_3(3)$, $M_{22}$ and $Aut(M_{22})$ over $\qq$.
All these examples make use of families of covers of the projective line ramified over four or more points, and therefore use techniques of explicit computations of Hurwitz spaces.
Similar techniques were used previously e.g.\ by Malle (\cite{Ma1}), Couveignes (\cite{Cou}), Granboulan (\cite{G}) and Hallouin (\cite{H}).
Unlike previous examples, however, some of our computations show the existence of rational points on Hurwitz spaces that would not have been obvious from theoretical arguments.}}
\maketitle
\section{Introduction}
In recent years, there has been a great deal of progress in explicit computation of polynomials with prescribed Galois group. One notable area of interest is the computation of $3$-point
covers of the line (Belyi maps),
for which strong tools have been developed, e.g.\ in \cite{KMSV}. Such techniques have been used to calculate explicit polynomials for many permutation groups of small degrees. Often
the existence of such polynomials defined over $\qq$ could a priori be deduced by the Rigidity Method (cf.\ \cite[Chapter I]{MM}). 
However, even for almost simple groups of relatively small degree, not all questions can be answered merely via $3$-point covers.\par
Meanwhile, covers with more than three branch points have been computed to solve some of those problems, like finding totally real polynomials with given Galois group, but also because they
sometimes give rise to multi-parameter polynomials over $\mathbb{Q}$.
A spectacular result in the computation of covers with more than three branch points was Granboulan's explicit $M_{24}$-polynomial in \cite{G}.
An important source for examples of multi-parameter polynomials is Malle's paper \cite{Ma1}, which also, along with Couveignes' \cite{Cou} and \cite{Cou2}, outlines methods for their calculation.
\par
The computational results of this article can be largely divided into two areas: the calculation of explicit polynomials for some of the almost simple groups of smallest permutation degree 
for which no polynomials were previously known; and the calculation of the first totally real polynomials for other almost simple groups.\par 
Table \ref{overview} summarizes very briefly the basic features of the families of polynomials occurring and of the Hurwitz spaces that they are parametrized by.
\begin{table}
\begin{tabular}{ccccccp{35mm}}
Section & $G$ &degree& \#branch points & Hurwitz variety & Real fibers & Remarks\\ \hline
\S\ref{psl52}&$PSL_5(2)$&$31$&$4$&$\mathbb{P}^1_{\alpha}$&no&\\

\S\ref{psl34}&$P\Gamma L_3(4)$&$21$&$4$&$\mathbb{P}^1_{\alpha}$&no&\\
&$PGL_3(4)$&$21$&$4$&&no& derived from previous; \newline base curve not generically $\mathbb{P}^1$\\
&$PSL_3(4)$&$21$&$5$&&no& see previous\\
\S\ref{tr_psl34}&$P\Gamma L_3(4)$&$21$&$4$&$\mathbb{P}^1_{\alpha}$&yes&\\
\S\ref{totally_real1}.1&$PGL_2(11)$&$22$&$4$&rank 1 ell.\ curve&yes&\\
&$PSL_2(11)$&$11$&$4$&&yes&\\
\S\ref{totally_real1}.2&$PGL_2(11)$&$12$&$4$&rank 1 ell.\ curve&yes&\\
&$PSL_2(11)$&$11$&$5$&&yes&\\
\S\ref{totally_real2}&$PSL_3(3)$&$13$&$5$&$\mathbb{P}^2_{\alpha,\beta}$&yes&\\
\S\ref{totally_real3}&$Aut(M_{22})$&$22$&$4$&rank 1 ell.\ curve&yes&\\
&$M_{22}$&$22$&$4$&&yes&derived from previous;\newline base curve $\mathbb{P}^1$\\
\end{tabular}
\label{overview}
\caption{Overview of the polynomials computed in this article}
\end{table}
Here, a Hurwitz variety $\mathbb{P}^1_{\alpha}$ leads to a one-dimensional family of covers, parameterized by $\alpha$, of the projective line $\mathbb{P}^1_t$, and therefore a two-parameter
polynomial $f(\alpha,t,x) \in \qq(\alpha,t)[x]$ with the prescribed Galois group. Similarly, the  Hurwitz variety $\mathbb{P}^2_{\alpha,\beta}$ in Section \ref{totally_real2} 
leads to a three-parameter polynomial. In the ``elliptic-curve" cases, one obtains the existence of an infinite family $f_P(t,x)$ of one-parameter polynomials, parameterized by the rational points $P$ of a rank-1
elliptic curve; for the sake of simplicity, only sample polynomials of these families are given. In some cases, polynomials for normal subgroups are derived in a natural way from polynomials with a given group. In these cases, it is to be understood in Table \ref{overview} that the Hurwitz variety
parameterizing the family of covers is the same as for the original group. 
\par
It should be noted that in several of the cases in Table \ref{overview} the precise nature of the Hurwitz variety and the base curves of the covering maps became clear only via explicit 
computation. In particular, the existence of rational points on the Hurwitz space as well as the rationality of the base curve of the respective covers - both necessary to obtain regular Galois 
realizations - was not always clear a priori. This will be addressed in more detail in Sections \ref{psl52}-\ref{totally_real3}. 
Before this, we will outline the theoretical background and the general techniques used for the computations.

\section{Theoretical background}
\label{RIGP}
We recall some basic facts about monodromy of covers, Hurwitz spaces and braid group action. For a deeper introduction, cf.\ \cite{FV}, \cite{RW} or \cite{Voe}.
\subsection{Covers of the projective line}
Let $S=\{p_1,...,p_r\}$ be a finite subset of the projective line $\mathbb{P}^1\cc$, $p_0\in \mathbb{P}^1\cc\setminus S$, and $f:R\to \mathbb{P}^1\cc\setminus S$ an $n$-fold covering map.
Then the fundamental group $\pi_1(\mathbb{P}^1\cc\setminus S, p_0)$ acts on the fiber $f^{-1}(p_0)$ via lifting of paths. This yields a homomorphism of $\pi_1(\mathbb{P}^1\cc\setminus S, p_0)$ into $S_n$,
and if $\gamma_i$ are homotopy classes of closed paths from $p_0$ around $p_i$ ($i=1,...,r$), ordered counter-clockwise in $\mathbb{P}^1\cc$, their images under this action, say 
$\sigma_1,...,\sigma_r$, generate a group isomorphic to the Galois group of $E\mid \cc(t)$, with $E$ being the Galois closure of the function field of (the compact Riemann surface) $R$.
Furthermore, we have $\sigma_1\cdots \sigma_r = 1$. We call $(\sigma_1,...,\sigma_r)$ the {\it branch cycle description} of the cover $f$.
The genus $g$ of $R$ is given by the Riemann-Hurwitz genus formula
$$g= -(n-1) + \frac{1}{2}\sum_{i=1}^r ind(\sigma_i),$$
where the index $ind(\sigma_i)$ is defined as $n$ minus the number of cycles of $\sigma_i\in S_n$.
This motivates the following definition:
\begin{defi}[Genus-$g$ tuple]
Let $G\le S_n$ be a transitive permutation group, $r\in \nn$ and $\sigma_1,...,\sigma_r \in G$ such that $\langle\sigma_1,...,\sigma_r\rangle = G$ and $\sigma_1\cdots \sigma_r=1$.
Then $(\sigma_1,...,\sigma_r)$ is called a genus-$g$ tuple of $G$, with $g:=-(n-1) + \frac{1}{2}\sum_{i=1}^r ind(\sigma_i)$.
\end{defi}

\subsection{Hurwitz spaces}
Let $G$ be a finite group. Let $S$ be a subset of the projective line $\mathbb{P}^1\cc$ of cardinality $r$, $p_0$ be any point in $\mathbb{P}^1\setminus S$
 and $f:\pi_1(\mathbb{P}^1\setminus S, p_0) \to G$ be an epimorphism mapping none of the canonical generators $\gamma_1,...,\gamma_r$ of the fundamental group to the identity. 
 On the set of such triples $(S,p_0,f)$ one defines an equivalence relation via $(S,p_0,f) \sim (S', p_0', f') :\Leftrightarrow S=S'$ and there exists a path $\gamma$ from $p_0$ to $p_0'$ in 
 $\mathbb{P}^1\setminus S$ such that the induced map $\gamma^\star: \pi_1(\mathbb{P}^1\setminus S, p_0) \to \pi_1(\mathbb{P}^1\setminus S, p_0')$ on the fundamental groups fulfills 
 $f'\circ \gamma^\star = f$.
Identifying the group $G$ with the deck transformation group of a Galois cover $\varphi: X \to \mathbb{P}^1\setminus S$, Riemann's existence theorem leads to a natural identification of 
these equivalence classes $[S,p_0,f]$ with equivalence classes $[\varphi,h]$, where $\varphi: X \to \mathbb{P}^1\setminus S$ is a Galois cover that can be extended to a branched cover of 
$\mathbb{P}^1$ with exactly $r$ branch points, and $h$ is an isomorphism from the group of deck transformations of $\varphi$ to $G$. Cf.\ \cite[Section 1.2.]{FV} 
(especially for the precise identification between the two different sets of equivalence classes) and \cite[10.1]{Voe}.\par
Denote the set of these equivalence classes by $\mathcal{H}^{in}_r(G)$.
This space carries a natural topological structure, and also the structure of an algebraic variety.
This directly links the inverse Galois problem with the existence of rational points on certain algebraic varieties.
The main result is the following (cf.\  \cite[Cor.\ 10.25]{Voe} and \cite[Th.\ 4.3]{DF2}):
\begin{satz}
\label{hurwitz_points}
Let $G$ be a finite group with $Z(G)=1$. There is a universal family of ramified coverings $\mathcal{F}:\mathcal{T}_r(G) \to \mathcal{H}_r^{in}(G)\times \mathbb{P}^1\cc$,
such that for each $h\in \mathcal{H}_r^{in}(G)$, the fiber cover $\mathcal{F}^{-1}(h) \to \mathbb{P}^1\cc$ is a ramified Galois cover with group $G$.
This cover is defined regularly over a field $K \subseteq \cc$ if and only if $h$ is a $K$-rational point. In particular, the group $G$ occurs regularly as a Galois group over $\mathbb{Q}$ 
if and only if $\mathcal{H}_r^{in}(G)$ has a rational point for some $r$.
\end{satz}
Monodromy action leads to a group theoretic interpretation of the above equivalence classes of covers.
\begin{defi}[Nielsen class]
\label{Nielsen_class}
Let $G$ be a finite group, $r\ge 2$ and $$\mathcal{E}_r(G):=\{(\sigma_1,...,\sigma_r) \in (G\setminus\{1\})^r \mid \sigma_1 \cdot ... \cdot \sigma_r = 1, \langle \sigma_1,...,\sigma_r\rangle=G\}$$ 
the set of all generating $r$-tuples in $G\setminus\{1\}$ with product 1. Furthermore let $\mathcal{E}_r^{in}(G)$ be the quotient of $\mathcal{E}_r(G)$ modulo conjugating the tuples simultaneously 
with elements of $G$.\par
For any $r$-tuple $C:=(C_1,...,C_r)$ of non-trivial conjugacy classes of $G$ the Nielsen class $Ni(C)$ is defined as the set of all $(\sigma_1,...,\sigma_r)\in \mathcal{E}_r(G)$ such that for 
some permutation $\pi \in S_r$ it holds that $\sigma_i \in C_{\pi(i)}$ for all $i \in \{1,...,r\}$. 
The definition of $Ni^{in}(C)$ is then possible in analogy to the above notation.
\end{defi}
Denote by $\mathcal{H}_r$ the Hurwitz braid group on $r$ strands. This group, a quotient of the Artin braid group, can be defined as the group
generated by $r-1$ elements $\beta_1,...,\beta_{r-1}$ fulfilling the classical braid relations 
$$\beta_i\beta_j=\beta_j\beta_i \quad \text{ for } 1\le i<j-1\le r-2,$$
$$\beta_{i}\beta_{i+1}\beta_{i}=\beta_{i+1}\beta_{i}\beta_{i+1} \quad\text{ for } 1\le i\le r-2$$
and the additional relation $$\beta_1\cdots\beta_{r-1}\beta_{r-1}\cdots \beta_1 = 1$$
(cf.\ \cite[Chapter III.1.1 and III.1.2]{MM}).
The group $\mathcal{H}_r$ acts naturally on the set $\mathcal{E}_r(G)$ (with an induced action on $\mathcal{E}_r^{in}(G)$) via
\begin{equation}
\label{braid_action}
(\sigma_1,...,\sigma_r)^{\beta_i} := (\sigma_1,...,\sigma_{i-1}, \sigma_i\sigma_{i+1}\sigma_i^{-1},\sigma_{i},...,\sigma_r), \text{\ for } i=1,...,r-1.
\end{equation}
It is obvious that the sets $Ni^{in}(C)$ are unions of orbits under these actions.\par
Furthermore, if $\mathcal{U}_r$ denotes the space of $r$-sets in $\mathbb{P}^1\cc$ and $\Psi: \mathcal{H}^{in}(G)\to \mathcal{U}_r$ is the branch point reference map,
the elements of a given fiber are in 1-1 correspondence with elements of $\mathcal{E}_r^{in}(G)$. 
Indeed, the above action on equivalence classes of $r$-tuples of elements of $G$ is, via this correspondence, essentially the same as the action of the fundamental group on the fiber via lifting of paths.
Each of the orbits of the braid group acting on $Ni^{in}(C)$ corresponds to a connected component of $\mathcal{H}^{in}_r(G)$. 
The union of all connected components corresponding to $Ni^{in}(C)$ is what is usually referred to as a Hurwitz space:
\begin{defi}[Hurwitz spaces]
For an $r$-tuple $C$ of conjugacy classes of a group $G$ with a non-empty Nielsen class $Ni^{in}(C)$, the union of components of $\mathcal{H}^{in}_r(G)$ corresponding to $Ni^{in}(C)$ is 
called the (inner) Hurwitz space of $C$. 
\end{defi}
If one leaves out the permutation $\pi$ in the above definition of a Nielsen class, one gets the notion of a straight Nielsen class: 
\[SNi(C):=\{(\sigma_1,...,\sigma_r) \in \mathcal{E}_r(G) \mid \sigma_i \in C(i) \text{\ for } i=1,...,r\}\]
The definition of $SNi^{in}(C)$ is then possible in analogy to Def.\ \ref{Nielsen_class}.\par
Now always assume that $Z(G)=\{1\}$, and that the braid group action on $SNi^{in}(C)$ is transitive.\footnote{This condition assures that the Hurwitz space is an absolutely irreducible variety over its 
field of definition. But even in the case of intransitive braid group action, there may still be an absolutely irreducible component, granted that there is a ``rigid" braid orbit, 
 e.g.\ a unique orbit of a given length.} 
Following \cite[Theorem 4.3]{DF2}, one has the following morphisms between (quasi-projective) varieties:
\begin{itemize}
 \item $\mathcal{F}: \mathcal{T} \to \mathcal{H}^{in}(C) \times \mathbb{P}^1$, the universal family of covers in the Nielsen class $Ni^{in}(C)$.
 \item $\mathcal{H}^{in}(C) \to \mathcal{U}_r$, mapping each point of $\mathcal{H}^{in}(C)$ to its set of branch points. 
 \item Proceeding to the pullback $(\mathcal{H}^{in})'(C):=\mathcal{H}^{in}(C) \times_{\mathcal{U}_r} \mathcal{U}^r$, one also obtains a morphism $(\mathcal{H}^{in})'(C) \to \mathcal{U}^r$, 
 with $\mathcal{U}^r$ the space of {\it ordered} $r$-sets in $\mathbb{P}^1\cc$.
 \item Via $PGL_2$-action, $(\mathcal{H}^{in})'(C)$ is birationally equivalent to $\mathbb{P}^1\cc \times \mathbb{P}^1\cc \times \mathbb{P}^1\cc \times \mathcal{H}^{red}(C)$, where 
 $\mathcal{H}^{red}(C)$ is the image under the above map of the subvariety of $(\mathcal{H}^{in})'(C)$ consisting of covers with the first three branch points equal to $0$, $1$, and $\infty$ 
 (in this order).
\item This restriction gives a morphism of $r-3$-dimensional varieties $\mathcal{H}^{red}(C) \to \mathcal{U}^{r-3}$.
\end{itemize}
Particularly in the case $r=4$, $\mathcal{C}:=\mathcal{H}^{red}(C)$ is a curve - it corresponds, via action of $PGL_2(\mathbb{C})$,
to the set of all covers with branch cycle description in $C$ and ordered branch point set $(0,1,\infty,\lambda)$, for some $\lambda\in \cc\setminus\{0,1\}$ (Of course,
this choice of branch points cannot always be assumed for covers defined over $\mathbb{Q}$; therefore one may consider covers with partially symmetrized branch point sets as well - cf.\ Chapter
III.7 in \cite{MM}).
The existence of Galois covers defined over a field $K$ is therefore directly linked to the existence of 
$K$-points on such curves (often called reduced Hurwitz spaces). We also refer to these reduced Hurwitz spaces as Hurwitz curves.
There are well known theoretical criteria to determine the genus of these Hurwitz curves, cf.\ e.g.\ Thm. III.7.8 in \cite{MM}.

\section{Computational methods}
\label{Comp}
\subsection{Deformation of genus zero covers}
\label{deform}
Let $Ni(C)$ be a Nielsen class of genus zero $4$-tuples generating a finite group $G$ (assume always $Z(G)=\{1\}$). 
Recall from Section \ref{RIGP} that, if $SNi^{in}(C)$ contains a unique orbit of length $n$ under the action of the braid group, $\mathcal{H}$ is the corresponding connected component of the (inner) 
Hurwitz space and $\mathcal{H}'$ its pullback over $\mathcal{U}^4$, then there is a natural degree-$n$ cover  $\mathcal{H}' \to \mathcal{U}^4$, where $\mathcal{H}'$ is birationally equivalent to 
$\mathcal{C} \times (\mathbb{P}^1\mathbb{C})^3$, and a degree-$n$ cover $\mathcal{C} \to \mathbb{P}^1\mathbb{C}$ of (irreducible projective non-singular) curves. 
If, via Moebius transformations, one fixes three of the four branch points of the genus zero covers, say to $0$, $1$ and $\infty$, one obtains a family of branched covers 
$\mathcal{T}_0 \to \mathcal{C} \times \mathbb{P}^1\mathbb{C}$. Let $t$ be a parameter for the projective line on the right side, then this family will have ordered ramification locus in $t$: 
$(0,\lambda,1,\infty)$, where $\lambda$ is a function on $\mathcal{C}$. 
As $\mathcal{C}$ is an irreducible curve, its function field is of one variable (and of degree $n$ over $\cc(\lambda)$), i.e.\ equal to $\cc(\lambda, \alpha)$ for some function $\alpha$.
Therefore the family $\mathcal{T}_0 \to \mathcal{C} \times \mathbb{P}^1\mathbb{C}$ can be expressed by a polynomial equation $f(\lambda,\alpha, t, X)=0$, where $f \in \cc(\lambda, \alpha)[t,X]$
is linear in $t$ (because of the genus zero condition).
For every specialization $t\mapsto t_0$ (e.g.\ to a ramification point), the coefficients of $f(\lambda,\alpha, t_0, X)$ lie in the function field $\cc(\lambda, \alpha)$.
To determine these coefficients, embed $\cc(\lambda)$ into the Laurent series field $\cc((\lambda))$. Then, using the fact that the finite extensions of $\mathbb{C}((\lambda))$ are all equal 
to some $\mathbb{C}((\mu))$ with $\mu^e=\lambda$, for some $e \in \mathbb{N}$ (cf.\  \cite[Chapter 2.1.3]{Voe}), all of these coefficients have a Puiseux expansion in $\lambda$, 
i.e.\ can be written as a Laurent series in $\mu:=\lambda^{\frac{1}{e}}$ with some $e\in \mathbb{N}$.
Here the exponent $e$ is nothing but the ramification index in the Hurwitz space of some place lying over $\lambda \mapsto 0$. 
This ramification index can be determined by group theoretical means: it is the number of equivalence classes of covers, i.e.\ of equivalence classes of $4$-tuples 
$(\sigma_1,\sigma_2,\sigma_3,\sigma_4)$ in $SNi^{in}(C)$, that lead to the same degenerate cover, i.e.\ class triple $(\sigma_1\sigma_2,\sigma_3,\sigma_4)$, upon letting $\lambda$ converge to zero.
\par
There are two important cases for practical computations:
\begin{itemize}
 \item If one knows an explicit polynomial for some degenerate ($3$-point) cover with monodromy $(\sigma_1\sigma_2,\sigma_3,\sigma_4)$ as above, one can determine $e$ 
 and then develop Puiseux expansions to regain a cover with $4$ branch points.
  The idea is to gain a sufficiently good initial approximation and then use Newton iteration to develop the series.
 A point in a given fiber of the non-degenerate cover which converges to a multiplicity-$k$ point $X\mapsto x_0$ of the degenerate cover will be of the form $X\mapsto x_0 + O(\mu^{e/k})$.
 To reach the necessary precision of the initial approximation, one needs to determine the unknown first-order coefficient.
 This is achieved by finding equations for the ``opposite" degeneration with monodromy $(\sigma_1, \sigma_2,\sigma_3\sigma_4)$, corresponding to $\mu\to 0$ from the viewpoint of a new 
 parameter $s:=t/\lambda$.
  \par
 A detailed description of this method has been given by Couveignes in \cite{Cou}, and an explicit Magma algorithm
 is contained in \cite{Koe}. Compare also the examples in the later sections, especially Section \ref{psl52_deform}.
\item If one even knows an explicit polynomial for some non-degenerate ($4$-point) cover of the family (say, ramified in $t\mapsto (0,1,\infty,a)$ for some $a\in \cc\setminus\{0,1\}$), 
then by mapping the branch points of the family to $t \mapsto (0,1,\infty,a+\lambda)$ one can develop from an {\it unramified} point, 
i.e.\ actually obtain Laurent series in $\lambda$ for the above coefficients. 
As one starts from a non-ramified point on the Hurwitz space, there is also no concern of getting into the Hurwitz space of a wrong four-point family by deforming, 
so computations can be done modulo suitable primes (as one doesn't need to double-check the monodromy via numerical methods in $\cc$).
However, for groups of larger degree, one cannot expect to directly find a polynomial for a non-degenerate cover, as the corresponding system of equations becomes too complicated.
\end{itemize}
{\bf Remark:}
\begin{itemize}
\item[a)] Of course all this remains true for $r$-tuples with $r\ge 5$ as well. In this case one either has to increase the transcendence degree to get the full Hurwitz space, 
or work at first only with a curve on the Hurwitz space, by fixing $r-1$ branch points in $t$ (in the unsymmetrized case).
 \item[b)] So far, all considerations were made over $\mathbb{C}$. However, for suitable choice of the conjugacy classes in $Ni(C)$, the 
 corresponding Hurwitz space can sometimes be defined over $\qq$. The Puiseux expansion approach may therefore be carried out over an appropriate number field.
 \item[c)] The above condition on the ordered ramification locus in $t$ to be $t\mapsto (0,1,\infty,\lambda)$ corresponds to the unsymmetrized case; 
 analogously, suitable Moebius transformations lead to different symmetrized cases; e.g.\ in the $C_2$-symmetrized case one can w.l.o.g.\ consider all covers with ordered ramification locus 
 $(\{\text{ zeroes of } t^2-\lambda\}, 1,\infty)$, etc.
 \end{itemize}
%
%
\subsection{Finding algebraic dependencies}
\label{algdep}
Assume for simplicity that the reduced Hurwitz space (obtained from $\mathcal{H}^{in}(C)$ via $PGL_2$-action) for a given family of covers with $r$ branch points can be defined over 
$\qq$.\footnote{Otherwise one gets the analogous results over some number field $K$.}
As this reduced Hurwitz space is an $(r-3)$-dimensional algebraic variety, its function field has transcendence degree $r-3$. Therefore, any $r-2$ elements of this function field must fulfill 
a non-trivial algebraic equation over $\qq$. In particular, the coefficients of an equation $f(t,x)=0$ for the corresponding universal family of covers (cf.\ the following sections) 
are such elements. This enables one to obtain explicit equations defining the Hurwitz space over $\qq$.\par
Again, for sake of simplicity, assume $r=4$, then the function field extension corresponding to the reduced Hurwitz space cover is of the form $F:=\qq(\lambda,\alpha)|\qq(\lambda)$, 
with a function field $F$ of one variable. The Puiseux expansion approach has embedded $F$ into the Laurent series field $K((\lambda^{1/e}))$ (for a suitable $e\in \mathbb{N}$ and a suitable number field $K$).
There are now different ways to obtain dependencies between two coefficients $\alpha_1,\alpha_2$ of the model. 
Under certain additional conditions, it will be clear that $\qq(\alpha_1,\alpha_2)$ is already the full function field $F$ and therefore the algebraic dependency between $\alpha_1$ and $\alpha_2$ is 
actually a defining equation for the Hurwitz curve. E.g., if the braid group acts primitively on the given Nielsen class, 
then there is no intermediate field between $F$ and $\qq(\lambda)$, so $\alpha_1:=\lambda$ and $\alpha_2$ any coefficient not contained in $\qq(\lambda)$ will suffice. 
This is usually not the best try, as $[F:\qq(\lambda)] = |SNi^{in}(C)|$ is often considerably larger than some other degrees $[F:\qq(\alpha_i)]$ 
(see the next section for theoretical results on the gonality of $F$).\par
The following approaches will be used in the following sections to obtain algebraic dependencies (cf.\ also Section 5 of \cite{Cou}):
\begin{itemize}
 \item[1)] If the coefficients $\alpha_i$ are actually given as Laurent series in $\mu:=\lambda^{1/e}$, simply solve a system of linear equations in order to see whether $\alpha_1,\alpha_2$ fulfill a 
 polynomial equation of degrees $n_1,n_2$ respectively. As such an equation has $N:=(n_1+1)(n_2+1)$ unknowns, series need to be expanded to precision at least $\mu^N$ in order to obtain sufficiently 
 many equations via comparison of coefficients. 
\par
An explicit (and precise!) Laurent series expansion is usually difficult to obtain over $\qq$, as the coefficients grow quite rapidly.
Therefore this approach, at least for dependencies of high degrees, can often be only obtained modulo some prime.
\item[2)] Once the degrees for algebraic dependencies are known (or can be conjectured, e.g.\ after mod-$p$ reduction), 
the corresponding systems of linear equations can also be solved numerically for complex approximations, with many different specialized values for $\lambda$, 
instead of one high-order Laurent series in $\lambda$. 
\item[3)] Instead of solving approximate complex equations numerically, a mod-$p$ solution can be lifted to many different solutions in $\zz_p$. 
The algebraic dependencies can then be retrieved via interpolation. 
\item[4)] If the degrees are not too high, algebraic dependencies can be obtained from complex approximations via the LLL-algorithm (see \cite{LLL}): 
suppose that $\alpha_1$, $\alpha_2$ fulfill a rational polynomial equation of degrees $n_1$ and $n_2$ respectively, specializing $\alpha_1$ to a rational value will leave $\alpha_2$ 
in a number field of degree at most $n_2$ over $\qq$. With sufficient precision, we managed to retrieve the minimal polynomials for these specialized values of $\alpha_2$ for degrees $n_2$ up to 100. 
Again, repeating this for many (at least $n_1+1$) different specializations for $\alpha_1$ will allow interpolation to retrieve the original equation.
\end{itemize}
{\bf Remark:}\\
Especially for larger braid orbits, with braid genus $g>0$, it may not always be possible to directly find algebraic dependencies for {\it all} coefficients occurring in an equation for the universal 
family (as some of these dependencies may be of very large degree). Therefore, in order to check whether a rational solution of some algebraic equation really corresponds to a ``good" point on the 
Hurwitz space (and not to a point on the boundary with degenerate monodromy!) one may have to find this point by moving through the Hurwitz space using Newton iteration.
To do this, one can use the monodromy action of the Hurwitz braid group (as the fundamental group of the space $\mathcal{U}_r$) in order to gain, from an approximation for a cover with branch cycle 
description $(\sigma_1,...,\sigma_r)$, approximations for all covers with the same set of branch points and branch cycle description in the same braid orbit.
E.g., applying the braid $\beta_i$ to a given cover with ordered branch point set $(p_1,...,p_r)$ corresponds to switching the $i$-th and the $(i+1)$-th branch point by moving each of them by 180 degree
on the the disc around $\frac{p_{i}+p_{i+1}}{2}$ with radius $|\frac{p_{i}-p_{i+1}}{2}|$ (assuming this disk contains no other branch points). See \cite[Lemma 10.9]{Voe}.
%
\subsection{Considerations about the gonality of function fields}
Usually the algebraic dependencies $f(a,b)=0$ will not be optimal with regard to the degrees of the variables $a,b$ involved. One can therefore use considerations about the gonality of the function field $K(a,b)$, 
involving computations of Riemann-Roch spaces, to find good parameters, i.e.\ rational function fields with low index in the function field $K(a,b)$.
This is especially useful in function fields of genus $0$ or $1$, or in hyperelliptic function fields.
\begin{defi}[Gonality]
Let $F|K$ be a function field of one variable. The gonality $gon(F|K)$ of $F|K$ is defined as the minimum of the degree $[F:K(x)]$ (for $x\in F$), i.e.\ the minimal index of a rational function field in $F$.
\end{defi}
We use the following estimates on the gonality of function fields, which also yield a method to explicitly find rational function fields $K(x) \subseteq F$ of low index.
\begin{lemma}
\label{gon}
 Let $g$ be the genus of the function field $F|K$. Then
\begin{enumerate}
 \item[a)] If $g=0$, then $gon(F|K) \le 2$.
 \item[b)] If $g\ge 2$, then $gon(F|K) \le 2g-2$.
 \item[c)] If $F|K$ has a prime divisor of degree one, then $gon(F|K) \le g+1$.
 \item[d)] If in addition $g\ge 2$, then $gon(F|K) \le g$.
\end{enumerate}
\end{lemma}
See \cite[Lemma 6.6.5]{J} for the proof. In each of the cases of Lemma \ref{gon}, computation of suitable Riemann-Roch spaces yields explicit elements $x\in F$ with $[F:K(x)]$ at most the bound given in the 
respective case.
\subsection{Galois group verification}
Once an exact polynomial equation (over $\mathbb{Q}$ or another number field) for a member of a given family of covers - or even for the entire family - has been found, it is necessary to 
verify the Galois group, especially considering that significant parts of the computations were based on numerical approximations. There are several easy ways to gain evidence for the Galois 
group. One of these is the computation of the monodromy by numerical means; this is a solid tool, although not an exact method - and
turning it into one requires considerable efforts.
However, in all the cases covered in the following sections, the structure of the Galois group allows for rigorous proofs, which are therefore given in detail. 
\par
The following sections will apply the theoretical and computational background to several examples of interest.
For each example, the structure will roughly follow the sequence of Sections \ref{RIGP} and \ref{Comp}:
firstly, a presentation of the properties of the Hurwitz family resp.\ braid orbit in question, followed by a description of the concrete techniques applied for deformation of covers and 
retrieving algebraic dependencies; finally, a presentation of explicit polynomials and verification of their Galois group.
\section{A family of polynomials with Galois group $PSL_5(2)$ over $\qq(t)$} 
\label{psl52}
We compute a family of coverings with four ramification points, defined over $\qq$, with regular Galois group $PSL_5(2)$.
This yields the (to my knowledge) first explicit polynomials with group $PSL_5(2)$ over $\qq(t)$. 
\subsection{A theoretical existence argument}
The group $PSL_5(2)$ does not have any rigid triples of rational conjugacy classes, and among the genus zero systems of rational class 4-tuples, there is only one with a Hurwitz curve of genus zero.
This curve will turn out to be rational in the course of the explicit computations, but this does not seem to be immediately clear by the standard braid orbit criteria (see below). 
However, if one looks at class 5-tuples, it is possible to obtain $PSL_5(2)$ as a regular Galois group over $\qq$ via purely theoretical arguments:
\begin{prop}
The inner Hurwitz space for the class 5-tuple $(2A,2A,2B,2B,3B)$ of $PSL_5(2)$ contains a rational curve over $\qq$, and therefore infinitely many $\qq$-points.
\end{prop}
\begin{proof}
This 5-tuple of classes arises as a rational translate of a 4-tuple of classes in $Aut(PSL_5(2))$. 
This 4-tuple (of classes $(2A,2B,2C,6A)$) has a single braid orbit of length 46; its Hurwitz curve is of genus zero, and the images of the braids in the action on this orbit fulfill
an oddness condition to guarantee the rationality of this genus zero curve. \par Every $\qq$-point of this rational curve realizes $Aut(PSL_5(2))$ regularly over $\qq$, 
and as the $PSL_5(2)$-fixed field of such a realization is a rational function field (of degree 2 over the base field), one also obtains $PSL_5(2)$.
\end{proof}
As the explicit computation of such a field extension requires the computation of $PSL_5(2)$-covers with 5 branch points, we content ourselves with a 4-point family in the following.
Note however, that the deformation methods of Section \ref{deform} could be used to obtain members of the above 5-point family from the 4-point one.
\subsection{Data of a Hurwitz family}
Let $G=PSL_5(2)$ in its natural permutation action on 31 points, and denote by $2A$ the class of involutions of cycle type $(2^8.1^{15})$, 
by $3B$ the class of elements of order 3 with cycle type $(3^{10}.1)$ in $G$, and by $8A$ the unique class of elements of order 8 in $G$ (of cycle type $(8^2.4^3.2.1)$).
We consider the straight Nielsen class $SNi(C)$ of class tuples of length 4, of type $(2A,2A,3B,8A)$ in $G=PSL(5,2)$, generating $G$ and having product $1$, i.e.\ 
\[SNi(C):=\{(\sigma_1,...,\sigma_4) \in G\mid \sigma_1,\sigma_2 \in 2A, \sigma_3\in 3B, \sigma_4 \in 8A, \langle \sigma_1,...,\sigma_4\rangle = G, \sigma_1 \cdots \sigma_4 = 1\}\]
In the notation of Section \ref{RIGP}, we have $|SNi^{in}(C)|=24$. The action of the braid group on $SNi^{in}(C)$, as given in Equation (\ref{braid_action}), 
is transitive and more precisely yields that there is a family of covers $\mathcal{T} \mapsto \mathcal{C} \times \mathbb{P}^1\mathbb{C}$, where $\mathcal{C}$ (the $C_2$-symmetrized reduced Hurwitz space) 
is an absolutely irreducible curve of genus zero, and for every $h \in \mathcal{C}$ the corresponding fiber cover is a Galois cover of $\mathbb{P}^1\mathbb{C}$ with Galois group $PSL_5(2)$.
\par
Although the usual braid genus criteria yield that the $C_2$-symmetrized Hurwitz space for this family is a genus-zero curve, it does not seem clear via standard theoretical considerations 
(e.g.\ odd cycle argument for the braid group generators, as in \cite[Chapter III. 7.4.]{MM}) whether it can also be defined as a rational curve over $\qq$. 
In particular, the cycle structure of the braid orbit generators acting on the Nielsen class does not yield any places of odd degree. 
More precisely, the image of the braid group is imprimitive on the 24 points, with 12 blocks of length 2 (i.e.\ if $F|\qq(t)$ is the corresponding function field extension, of degree 24, 
we have an inclusion $\qq(t)\subset E \subset F$, with $[E:\qq(t)]=12$ and $[F:E]=2$). As the images in the action on the blocks of the three braids defining the ramification structure of these fields 
have cycle structure $(4^2.3.1)$, $(7.3.2)$ and $(2^5.1^2)$ respectively, it is clear that $E$ is still a rational function field; however the cycle structure of the latter involution in the action on
24 points is $(2^{12})$, so it is possible that a degree-2 place of $E$ ramifies in $F$, in which case the rationality of $F$ is not guaranteed.\footnote{Closer group theoretic examination yields 
some evidence for prime divisors of odd degree: namely, the two $3$-cycles of the braid group generator of cycle structure $(7^2.3^2.2^2)$ correspond to degenerate covers with three ramification points,
generating two isomorphic, but {\it non-conjugate} (in $PSL_5(2)$) subgroups. The same holds for the two $2$-cycles of this braid group generator. 
The explicit computations show that the corresponding prime divisors of ramification index 3 and 2 respectively are indeed of degree 1.}
We therefore clarify the rationality of this curve by explicit computation.
\subsection{Deformation of covers}
\label{psl52_deform}
We start with a degenerate cover with ramification structure $(2A,21A,8A)$, with group $PSL_5(2)$. We solve the corresponding system of equations for the three-point cover modulo 11, 
and then lift and retrieve algebraic numbers from the $11$-adic expansions.
The triple is rigid, but as the conjugacy class of the element of order 21 is not rational, we obtain a solution over a quadratic number field, namely
\[0=x^{21}\cdot(x-1)^7\cdot(x-a_1)^3-t\cdot(x^2-2\cdot x+a_2)^8\cdot(x^3-2\cdot x^2+a_3\cdot x+a_4)^4\cdot(x-a_5),\]
where 
$(a_1,...,a_5) := (\frac{1}{8}(-\sqrt{-7} + 11), \frac{1}{16}(-\sqrt{-7} + 11),\frac{1}{16}(\sqrt{-7} + 21), \frac{1}{128}(-3\sqrt{-7} - 31),\frac{1}{8}(-\sqrt{-7} + 3)
)$.
\par
From this degenerate cover, we develop complex approximations for a cover branched in four points, using Puiseux expansions as outlined in Section \ref{deform}.
As pointed out there, in order to turn the above 3-point cover into a first-order approximation of the 4-point family, we need to consider the ``opposite" degeneration as well.
Therefore, write the element of order $21$ as a product of two elements $\sigma_2$ of class $2A$ and $\sigma_3$ of class $3B$. One verifies that in all cases, the triple 
$(\sigma_2,\sigma_3,(\sigma_2\sigma_3)^{-1})$ generates an intransitive group isomorphic to $PSL_3(2)\times C_3$, with orbits of length 21, 7 and 3. 
Equations for the genus zero covers induced by this triple on each orbit are easily computed 
(especially since the degree 21 action is imprimitive, so the corresponding equation arises as a composition of functions of degree 3 and 7) and yield all the information needed for first-order
approximations.
\par
Now let $\mathbb{C}(x)|\mathbb{C}(t)$ be the field extension of rational function fields corresponding to the cover with four branch points. Via Moebius transformations (in $x$ and in $t$) it is 
possible to assume a defining polynomial \[f:=f(t,x):=f_0(x)^3\cdot (x-3)-t\cdot g_0(x)^8\cdot g_1(x)^4\cdot x,\] where $\deg(f_0) = 10$, $\deg(g_0)=2$ and $\deg(g_1)=3$
(so we have e.g.\ assumed the element of order 8 to be the inertia group generator over infinity, and the element of order 3 the one over zero).
Also, assume that for some $\lambda \in \cc$ the polynomials $f_a:=f(a,x)$ and $f_b:=f(b,x)$
(where $a$ and $b$ shall denote the complex zeroes of $x^2+x+\lambda$)
become inseparable in accordance with the elements in the conjugacy class $2A$ .
\par
Once we have obtained a complex approximation of such a polynomial $f$,
we now slowly move the coefficient at $x^2$ of the above polynomial $g_1$ to a fixed rational value, and apply Newton iteration to expand the other coefficients with sufficient 
precision to then retrieve them as algebraic numbers (using the LLL-algorithm). 
One finds that all the remaining coefficients come to lie in a cubic number field. For example, specialization to the rational value $-1$ leads to a root field of 
$x^3 - 14x^2 - 22x - 16$, as can be verified with the values in Theorem \ref{psl52pol}.
This already indicates that there is a rational function field of index 3 
 in the (genus-zero) function field of the Hurwitz space, which would enforce the latter function field to be rational over $\qq$ as well. This will be confirmed by the further computations.
\subsection{Algebraic dependencies and exact equations}
We now choose a prime $p$ such that the above solution, found over a cubic number field, reduces to an $\mathbb{F}_p$-point. Any prime such that the defining polynomial of the cubic
number field has a single root modulo $p$ will do, e.g.\ $p=11$ for our example.
Then we apply approach no.3 described in Section \ref{algdep}, that is, we lift this point to sufficiently many $p$-adic solutions (all coalescing modulo $p$), in order to obtain algebraic 
dependencies between the coefficients\footnote{Alternatively, one could just repeat the process of rational specialization and Newton iteration, as above, sufficiently often, 
obtaining cubic minimal polynomials for the other coefficients in each case, and then interpolate.}.
These dependencies are all of genus zero, and luckily some of them are of very small degree, e.g.\ if $c_2$ and $c_1$ are the coefficients at $x^2$ resp.\ $x$ of the polynomial $g_1$, 
one obtains an equation \[\sum_{i=0}^2\sum_{j=0}^3 \alpha_{ij} c_2^i c_1^j =0\] of degrees 2 and 3 respectively.
As there are a priori $(2+1)\cdot (3+1)=12$ unknown coefficients $\alpha_{ij}$, 
we only need 12 different $p$-adic liftings to find this dependency as the smallest degree dependency between $c_1$ and $c_2$ - and maybe a few more to gain evidence that 
it is not a coincidence. Of course, we find $\alpha_{ij}\in \qq_p$, but for theoretical reasons we expect them to actually be rational numbers - and indeed it is easy to retrieve the 
actual rational numbers from a sufficiently close $p$-adic approximation.
Next, one easily finds a parameter $\alpha$ for the rational function field defined by this equation, using Riemann-Roch spaces (cf.\ Lemma \ref{gon}). 
\par
Now, we can express all coefficients as rational functions in $\alpha$, and obtain the following result:
\begin{satz}
\label{psl52pol}
Let $\alpha, t$ be algebraically independent transcendentals over $\qq$.\\
Define polynomials $f_0,g_0,g_1 \in \qq(\alpha)[x]$ as follows:
{\footnotesize
\[f_0:= (x^5-2\frac{(\alpha+1)(\alpha+4)}{\alpha-2}x^4-2\frac{(\alpha+1)(\alpha^3 - 15\alpha^2 - 6\alpha - 152)}{(\alpha-2)(\alpha+4)}x^3\]
\[+8(\alpha+1)(\alpha^2-\alpha+7)x^2-7\frac{(\alpha+1)^2 (\alpha^3 + 12/7 \alpha^2 + 3/7 \alpha + 106/7)}{\alpha-2}x +2\frac{(\alpha+1)^5(\alpha+4)}{\alpha-2})\]
\[\cdot (x^5+4\frac{(\alpha-5)(\alpha^2 + 5/4 \alpha + 19/4)}{(\alpha+1)^2}x^4-2\frac{\alpha^3 + 42 \alpha^2 + 45 \alpha + 220}{\alpha+4}x^3\]
\[-12\frac{(\alpha+1)(\alpha^4 - 5/2 \alpha^3 - 27/2 \alpha^2 - 29 \alpha - 100)}{(\alpha-2)(\alpha+4)}x^2+9\frac{(\alpha+1)^2 (\alpha^3 + 8/3 \alpha^2 + 19/3 \alpha + 50/3)}{\alpha-2}x-3(\alpha+1)^4),\]
\[g_0:= x^2-6x-(\alpha+1)^2,\]
\[g_1:= (x-\frac{(\alpha+1)(\alpha+4)}{\alpha-2}) \cdot (x^2+2\frac{(\alpha-2)(\alpha+1)}{\alpha+4}x -(\alpha+1)^2).\]}
Then the polynomial $f(\alpha,t,x):= f_0^3 \cdot (x-3) - t \cdot g_0^8 g_1^4 \cdot x$, of degree 31 in $x$, has Galois group $PSL_5(2)$ over $\qq(\alpha,t)$, with ramification structure
$(2^8.1^{15}, 2^8.1^{15}, 3^{10}.1, 8^2.4^3.2.1)$ with respect to $t$.
\end{satz}
\begin{proof}
Dedekind reduction, together with the list of primitive groups of degree 31 (as implemented e.g.\ in Magma), shows that $PSL_5(2)$ must be a subgroup of the Galois group. 
It therefore suffices to exclude the possibilities $A_{31}$ and $S_{31}$.\par
Multiplying $t$ appropriately, we can assume the covers to be ramified in $t=0, t=\infty$ and the zeroes of $t^2+t+\lambda$, with some parameter $\lambda$. Interpolating through sufficiently many values of $\alpha$ one 
finds the degree-24 rational function $\lambda=\frac{h_1(\alpha)}{h_2(\alpha)}$ parameterizing the Hurwitz curve. As e.g.\ $\alpha=0$ and $\alpha=1/2$ yield the same value for $\lambda$, we set 
$t=C \cdot (\frac{f_0^3 \cdot (x-3)}{g_0^8\cdot g_1^4\cdot x})(0,s)$ (evaluating $x$ to a parameter $s$ of a rational function field, as well as $\alpha$ to $0$, and multiplying with a suitable constant $C$ to obtain 
the above condition on the branch points). Then one can check that over $\qq(s)$, the polynomial $f(1/2, C_2\cdot t,x)$ (again for a suitable constant $C_2$ to obtain the branch point conditions) splits into two factors 
of degrees 15 and 16. This means that for this particular point of the family, there is an index-31 subgroup of the Galois group that acts intransitively on the roots. As $PSL_5(2)$ has such a subgroup and 
$A_{31}$ and $S_{31}$ don't, the Galois group for this particular specialization is $PSL_5(2)$. This specialization corresponds to an unramified point on the (irreducible) Hurwitz space, therefore the entire family must 
belong to the same Hurwitz space and therefore have Galois group $PSL_5(2)$ over $\qq(\alpha,t)$.
\end{proof}
We can now specialize $\alpha$ to any value that does not let two or more ramification points coalesce, to obtain polynomials with nice coefficients with group $PSL_5(2)$ over $\qq(t)$.
E.g.\ $\alpha \mapsto 0$ leads to:
\begin{kor}
 The polynomial
\[\tilde{f}(t,x):=(x^5 - 95 x^4 - 110 x^3 - 150 x^2 - 75 x - 3)^3 (x^5 + 4 x^4 - 38 x^3 + 56 x^2 + 53 x - 4)^3(x-3)\]
\[ - t (x^2 - 6 x - 1)^8(x^2-x-1)^4 (x+2)^4 x \in \qq(t)[x]\]
defines a regular extension of $\qq(t)$ with Galois group $PSL_5(2)$.
\end{kor}
In fact it can be seen from $\lambda=\frac{h_1(\alpha)}{h_2(\alpha)}$ (as in the proof above) that the only specialized rational values for $\alpha$ that lead to degenerate covers 
(with two branch points coalescing) are $\alpha \mapsto -4$, $\alpha \mapsto -1$ and $\alpha \mapsto 2$.\\ \\
{\bf Remark:}\\
The above proof essentially uses the fact that $PSL_5(2)$ has two 
non-conjugate actions on 31 points inducing the same permutation character. This can of course be applied to other linear groups, and has e.g.\ been used in \cite{Ma1} to verify $PSL_2(11)$ 
(and others) as the Galois group of a family of polynomials. Cf.\ also the Galois group verifications in the following sections.

\section{Polynomials with Galois group $PSL_3(4) \le G\le P\Gamma L_3(4)$ over $\mathbb{Q}(t)$}
\label{psl34}
\subsection{Review of known results}
Previously, there have not been any explicit polynomials $f(t,X)\in \qq(t)[X]$ with regular Galois group $P\Gamma L_3(4) (=PSL_3(4).S_3), PGL_3(4) (=PSL_3(4).3)$ or $PSL_3(4)$.
Malle gave a polynomial for $PSL_3(4).2$ (the extension of $PSL_3(4)$ by the field automorphism) in \cite[Theorem 3]{Ma2}, but this does not yield a $PSL_3(4)$-polynomial, 
as the $PSL_3(4)$-fixed field does not have genus 0 
 (see however \cite[p.2]{Zyw} for a way to obtain from Malle's polynomial a $PSL_3(4)$-polynomial over $\qq$ (not $\qq(t)$).\par
Theoretical arguments for all $PSL_3(4)\le G\le P\Gamma L_3(4)$ to be a regular Galois group over $\mathbb{Q}(t)$ have however been known for a long time 
(cf.\ \cite{MM}, Example 4.2. in Chapter IV.4).
\subsection{Data of a Hurwitz family}
We find polynomials for all groups $PSL_3(4) \le G\le P\Gamma L_3(4)$ by computing the Hurwitz space of a family of covers with Galois group $P\Gamma L_3(4)$, ramified over four places with ramification structure 
$(2^7.1^7, 2^7.1^7, 3^5.1^6, 5^4.1)$ with regard to the natural degree 21 permutation representation of $P\Gamma L_3(4)$. 
The length of the corresponding Nielsen class is 20, and the $C_2$-symmetrized inner Hurwitz curve
 is a rational curve of genus zero. Therefore this family leads to many polynomials with regular Galois group $P\Gamma L_3(4)$ over $\mathbb{Q}$. The fixed field of $PSL_3(4)$ in such an extension is still of genus zero, 
 as can be seen by the coset action of the above class four-tuple on $P\Gamma L_3(4)/PSL_3(4)$. However, even the fixed field of $PGL_3(4)$ cannot be guaranteed to be a rational function field by theoretical means (it is 
 a genus-zero degree-2 extension of the function field $\mathbb{Q}(t)$, ramified in two places, which are possibly algebraically conjugate, in which case the extension field need not be rational).
The above fixed field would automatically be rational for any rational point on the {\it unsymmetrized} Hurwitz curve - i.e. for a regular $P\Gamma L_3(4)$-extension with all branch points rational -
but this curve is not of genus zero anymore.\par
We therefore verify by explicit computation that the fixed field of $PSL_3(4)$ is indeed a rational function field for suitable choices of parameters; this yields explicit polynomials with
regular Galois groups $PGL_3(4)$ and $PSL_3(4)$ as well.
\subsection{A family of polynomials with regular Galois group $P\Gamma L_3(4)$}
The deformation and algebraization process for our family is analogous to the one in Section \ref{psl52} (note that the Hurwitz curves are rational in both cases). It should therefore suffice to present the 
resulting polynomial. We only note briefly that a good permutation triple to start the deformation process from is the triple with cycle structures $(2^7.1^7, 8^2.4.1 ,5^4.1)$, generating
the transitive subgroup $PSL_3(4).2$. A polynomial for this triple is easily found modulo a small prime and then lifted to a polynomial defined over a number field - in this case 
$\mathbb{Q}(\sqrt{-1})$.
\begin{satz}
The polynomial 
\begin{adjustwidth}{-2cm}{-2cm}
{\footnotesize
\[f:=(x^3+(\alpha-10)x^2-(\alpha^2+20)x+5\alpha)^5 (x+1)^5 x \]
\[- t((\alpha^2-6\alpha+45) x^5+\frac{1}{8}(\alpha^4-6\alpha^3+85\alpha^2-132\alpha+1476)  x^4+\frac{1}{2}(\alpha^4-4\alpha^3+53\alpha^2-138\alpha+360) x^3\]
\[+\frac{1}{4}\alpha (\alpha^3-28\alpha^2+77\alpha-450) x^2+\frac{1}{2}\alpha^2 (\alpha^2-2\alpha+65)  x+\frac{1}{8}\alpha^2(\alpha-5)^2)^3\cdot\]
\[ (\alpha (\alpha+3)  x^5+(4\alpha^3-15\alpha^2+47\alpha+192)  x^4+2 (2\alpha^4-20\alpha^3+127\alpha^2-329\alpha+880)  x^3\]
\[+2  (2\alpha^4-36\alpha^3+347\alpha^2-1485\alpha+3000)  x^2+(-44\alpha^3+405\alpha^2-3325\alpha+9000)  x+125 (\alpha^2-5\alpha+40))\in \mathbb{Q}(\alpha,t)[x]\]}
\end{adjustwidth}
has regular Galois group $P\Gamma L_3(4)$ over $\mathbb{Q}(\alpha,t)$, with ramification structure $(2^7.1^7, 2^7.1^7, 3^5.1^6, 5^4.1)$ with regard to $t$.
\end{satz}
\begin{proof}
Specializing in appropriate finite fields, one sees that the Galois group of $f$ is either $P\Gamma L_3(4)$ or $S_{21}$.
Now $P\Gamma L_3(4)$ has two non-conjugate subgroups $U$ and $V$ of index 21. If one considers the action of $P\Gamma L_3(4)$ on the right cosets of $U$, then $V$ is intransitive with orbits 
of length 5 and 16. The images of the desired inertia subgroup generators $\sigma_1,...,\sigma_4$ in the action on the cosets of $V$ are still of the same cycle type, 
and therefore belong to the same family of covers, but not to the same cover. 
\par
A suitable linear transformation in $t$ assures that the function field extension $\qq(\alpha)(x)\mid \qq(\alpha)(t)$ is ramified over $t\mapsto 0$, $t\mapsto \infty$ and $t\mapsto \{$zeroes of $t^2+t+\mu(\alpha)\}$ for some
rational function $\mu(\alpha)\in \qq(\alpha)$. This choice of ramification yields a good model for the $C_2$-symmetrized Hurwitz curve. One then notes that the specializations $\alpha\mapsto 10$ and $\alpha\mapsto 13$ lead to the
same ramification locus. If $f_{10}(t,x)=p_{10}(x)-tq_{10}(x)$ and $f_{13}(t,x)=p_{13}(x)-tq_{13}(x)$ are the corresponding polynomials, the polynomial $p_{10}(x)\cdot q_{13}(y)-q_{10}(x)\cdot p_{13}(y)$
decomposes in $\qq[x,y]$ into factors of degree 5 and 16. This means that there is an index-21 subgroup in the Galois group of $f_{10}(t,x)$ acting intransitively with orbits of length 5 and 16.
Therefore $f_{10}$ must have Galois group $P\Gamma L_3(4)$, and as our Hurwitz space is connected, the same must hold for the two-parameter polynomial $f$.
\end{proof}
\subsection{Descent to proper normal subgroups of $P\Gamma L_3(4)$}
As noted above, the fixed field of $PGL_3(4)$ in the Galois closure of $f$ is of genus zero. It is given as $\qq(\alpha)(X,Y)$, 
where $p_\alpha(X,Y):=X^2+3(Y^2-(\alpha^2-15\alpha+90)(\alpha^2-5\alpha+40))=0$.
Although the conic given by $p_\alpha(X,Y)=0$ does not split generically - i.e.\ it does not have any $\mathbb{Q}(\alpha)$-rational points, 
there are many values $\alpha_0\in \mathbb{Q}$ for which the specialized curve given by $p_{\alpha_0}(X,Y)=0$ has non-singular points, which means that
the residue field $\qq(X,Y)[\alpha]/(\alpha-\alpha_0)$ is a rational function field for these values $\alpha\mapsto \alpha_0$, i.e.\ it can be parametrized as $\qq(s)$. 
One such example is $\alpha_0=10$.
In this case, parametrizing $t$ as a rational function in $s$ yields the following polynomial, with regular
Galois group $PGL_3(4)$ over $\qq$:
$$g:=(s^2+3)(x^3-120x+50)^5(x+1)^5 x - \frac{4}{3\cdot5^5\cdot13^2\cdot17^4}(111151s^2+389344s-55891)\cdot$$
$$(85x^5 + 1582x^4 + 5140x^3 - 3700x^2 + 7250x + 625/2)^3(130x^5 + 3162x^4 + 20580x^3 + 13700x^2 - 27750x + 11250)$$ $$ \in \qq(s)[x]$$
As the fixed field of $PSL_3(4)$ is a degree 3 genus zero extension of the fixed field of $PGL_3(4)$, it is a rational function field whenever the latter field is. 
Parameterizing it for our specialization $\alpha_0=10$ leads to the following polynomial with regular Galois group $PSL_3(4)$:
$$h:=(y^2-y+1)^3(x^3-120x+50)^5(x+1)^5 x$$
$$-(\frac{4}{751689}y^6 - \frac{4}{250563}y^5 - \frac{2783192}{132328584375}y^4 + \frac{27261652}{396985753125}y^3 - \frac{2783192}{132328584375}y^2 - \frac{4}{250563}y + \frac{4}{751689})\cdot$$
$$(85x^5 + 1582x^4 + 5140x^3 - 3700x^2 + 7250x + 625/2)^3(130x^5 + 3162x^4 + 20580x^3 + 13700x^2 - 27750x + 11250)$$ $$\in \qq(y)[x]$$
\subsection{Totally real extensions with group $PSL_3(4)\le G \le P\Gamma L_3(4)$}
\label{tr_psl34}
The family computed above does not yield any totally real Galois extensions with the above groups, as can be checked easily by observing the action of complex conjugation on the class tuples in
the Nielsen class. This conjugation is never given by the identity element of $P\Gamma L_3(4)$, which would however be necessary to obtain a totally real specialization.\par
On the other hand, the family used in \cite{MM}, Example 4.2 in Chapter IV.4 to obtain $PSL_3(4)\le G\le P\Gamma L_3(4)$ as regular Galois groups by theoretical means does lead to such specializations.
I computed this family in an earlier version of this paper (cf.\ \cite[Chapter 7]{Koe}); it also has a rational Hurwitz curve, but the corresponding polynomials turned out to have rather large coefficients, 
therefore I will only give a single polynomial for $P\Gamma L_3(4)$. Polynomials for proper normal subgroups can be obtained from this in the usual way.
\begin{satz}
The polynomial 
\[f(t,x):=(x^7 + \frac{18453672844570518827351}{464949935671}  x^6 - \frac{207994860612980146110393025396186540191}{2812713705941843158}  x^5 \]
\[+ \frac{28099474349691216216874520999969033907118201199}{1406356852970921579}  x^4 \]
\[+ \frac{21503029546831034221405520441479341846831570716278114576}{1406356852970921579}  x^3\]
\[- \frac{9875613161329199448867490608939590635407743468957241801995410272}{1406356852970921579}  x^2 \]
\[  + \frac{2934026230199894418359951279176917481405147836113333573421902044849280}{4672281903557879}  x\]
\[+ \frac{3220744414074178541841609287239948730104227472303051912345719011603335979776}{4672281903557879})^2 \cdot \]
\[(x^7 + \frac{6629673981088984}{10049}  x^6 - \frac{4334793194194588640311258112563598440086555375}{576034733687471381342032}  x^5 \]
\[+ \frac{19752423757662662431040186068639932484605957602986058305001}{4703215594045637427773689649}  x^4 \]
\[- \frac{14339959909531370924438660628225217373062780277511175009294545777430834}{47262613504564610511697807282801}  x^3 \]
\[- \frac{4636043913327361505565912998538088120702737210037510505278818632207867738346240}{47262613504564610511697807282801}  x^2\]
\[ + \frac{40249862706254613322713917502727163663607545615294770224178252169533070179559152324432}{47262613504564610511697807282801}  x \]
\[+ \frac{43051490625182263519732001313495514865873179734602004170000304283139316449024}{4634750377158001})\]
\[-t  (x + \frac{291343529284}{10049})^6  (x - 2349544591)^6  (x - 346425124)^3  (x - 304781764)^3  x\]
has regular Galois group $P\Gamma L_3(4)$ over $\qq(t)$. The ramification structure with regard to $t$ is of type $(2^7.1^7, 2^8.1^5, 3^5.1^6, 6^2.3^2.2.1)$.
Furthermore, let $a$ be the unique ramification point of $f$ inside the real interval $(-\infty, 0)$, i.e.\ $a \approx -8.75 \cdot 10^{22}$.
Then for $t_0\in (a,0)$, the specialized polynomial $f(t_0,x)$ is totally real.
\end{satz}

\section{Totally real extensions with group $PGL_2(11)$}
\label{totally_real1}
In this section and the following ones, we will focus on totally real extensions. In particular, we compute explicit polynomials for totally real Galois extensions over $\qq$, with Galois 
groups $PGL_2(11)=PSL_2(11).2$, $PSL_3(3)$, $M_{22}$ and $Aut(M_{22})=M_{22}.2$.
The first two of these groups are the smallest (with respect to minimal faithful permutation degree) that have not been previously realized as the Galois group of a totally real extension 
of $\qq$, this means that explicit totally real polynomials are now known for every transitive permutation group of degree at most 13 (cf.\ \cite{KM2}).
By now, some totally real specializations of the $PGL_2(11)$- and $PSL_3(3)$-polynomials computed below have been inserted in the database \cite{KM2}.\par
Note that totally real extensions can only be obtained via families with four or more branch points, cf.\ \cite{MM}, Chapter I, Example 10.2.
The problem for the group $PGL_2(11)$ is that, on the one hand, to obtain totally real fibers (i.e.\ a complex conjugation acting as the identity) 
 one needs to compute polynomials with at least four branch points. On the other hand $PGL_2(11)$ in its natural action has no generating genus zero tuples of length $r \ge 4$.
There are however genus zero tuples in the imprimitive action on 22 points, which stems from the exceptional action of $PSL_2(11)$ on 11 points (this degree-11 action was also used
by Malle to compute totally real $PSL_2(11)$-polynomials in \cite[Section 9]{Ma1}).
Below are explicit computations for two such class tuples.
\subsection{The ramification type $(2A,2B,2B,3A)$}
\subsubsection{Hurwitz data and assumptions on branch points}
Firstly, let $C=(2A,2B,2B,3A)$ the quadruple of classes of $PGL_2(11)$, where $3A$ is the unique class of elements of order 3, $2A$ is the class of involutions inside $PSL_2(11)$, and $2B$ the class of involutions outside $PSL_2(11)$. 
This is a genus zero tuple in the imprimitive action on 22 points, so for a degree-22 cover of $\mathbb{P}^1(\mathbb{C})$ with this ramification type, 
we get the following inclusion of function fields: $\cc(t) \subseteq \cc(s) \subseteq \cc(x)$, where exactly two places of $\cc(t)$ ramify in $\cc(s)$ 
(namely the ones with inertia group generator not contained in $PSL_2(11)$), and exactly four places of $\cc(s)$ ramify in $\cc(x)$ 
(namely two places lying over the ramified place of $\cc(t)$ with inertia group generator in $2A$, and two lying over the place of $\cc(t)$ with inertia group generator $3A$).
\par
The essential task is therefore to compute the extension $\cc(x)|\cc(s)$ , i.e. to compute polynomials with $PSL_2(11)$-monodromy, defined over $\qq$ if possible, and ramification type $(2A,2A,3A,3A)$.
The straight inner Nielsen class of these tuples in $PSL_2(11)$ is of length $|SNi^{in}|=54$, 
with transitive braid group action and symmetrized braid orbit genus $g_{12}=1$.\footnote{Additional symmetrization of the branch points 3 and 4 does not decrease this genus.}
Via Moebius transformations, we therefore assume that the two places of $\cc(s)$ with inertia group generator of order 3 are $s\mapsto 0$ and $s\mapsto \infty$, and also fix the sum of the other two branch points.
As the cycle structure of an element $\sigma$ in the class $3A$ of $PSL_2(11)$ in the action on 11 points is $(3^3.1^2)$, and one of the $3$-cycles remains fixed under conjugation with 
$N_{PSL_2(11)}(\langle\sigma\rangle)$ (and therefore under the action of the decomposition subgroup), one can assume w.l.o.g. for a model over $\qq$ that the place $x\mapsto 0$ lies over
$s\mapsto 0$ (with ramification index 3), and the same for $x\mapsto \infty$ and $s\mapsto \infty$.
That is, we may w.l.o.g.\ look for polynomial equations $x^3\cdot f_1(x)^3 \cdot f_2(x) - s\cdot g_1(x)^3 \cdot g_2(x)=0$, with quadratic polynomials $f_i,g_i$.
\subsubsection{Computations}
Due to the relatively small degree, one can immediately search for a mod-$p$ reduced polynomial with the above restrictions on places and the correct ramification,
instead of starting with a 3-point cover and going through the deformation process. There is a solution with the correct Galois group over $\mathbb{F}_7$.  
\par
Now lift this solution to many approximate $\mathbb{Q}_7$-solutions, with the set of zeroes of $s\cdot (s^2+4s+\lambda)$ as the finite ramification locus (for many different values of $\lambda$).
Interpolation then yields an algebraic dependency between the coefficients at $x^1$ of the above polynomials $g_1$ and $g_2$, namely:
$$(88/19\beta^2 - 112/19 \beta + 32/19)\alpha^4 + (178/19 \beta^3 - 524/19 \beta^2 + 446/19 \beta - 112/19) \alpha^3 $$
$$+ (287/38 \beta^4 - 650/19 \beta^3 + 2051/38 \beta^2 - 662/19 \beta + 295/38)\alpha^2 $$
$$+ (59/19 \beta^5 - 687/38 \beta^4 + 773/19 \beta^3 - 1675/38\beta^2 + 435/19 \beta - 173/38) \alpha$$
$$+ 10/19 \beta^6 - 70/19 \beta^5 + 21/2 \beta^4 - 595/38 \beta^3 + 491/38 \beta^2 - 213/38 \beta + 1=0$$
(with $\alpha$ the coefficient of $g_1$ and $\beta$ the one of $g_2$).
Computation with Magma confirms that this defines an elliptic curve of rank 1 (more precisely, this curve can be defined by the cubic equation $Y^2 = X^3 - 27X - 10$), which therefore has infinitely many points.
Furthermore all other coefficients of the model can be expressed as polynomials in $\alpha$ and $\beta$, therefore this curve is already a model of the reduced Hurwitz curve of the $PSL_2(11)$-family. 
So there are infinitely many equivalence classes of covers defined over $\qq$ with this monodromy.
\par
However, as we are interested in totally real polynomials, we need to choose a point on the curve in such a way, that complex conjugation on a fiber of the corresponding $PSL_2(11)$-cover 
is trivial in at least one segment of the punctured projective line.
Monodromy computations show that $\alpha=-\frac{3}{121}$ and $\beta=\frac{41}{55}$ yields such a point. This leads to the polynomial 
\begin{equation}
\begin{split}
\label{psl211_a}
f(s,x):=x^3(x^2+x-\frac{413}{4114})^3(x^2-\frac{23}{726}x+\frac{63}{181016})\\
-s(x^2-\frac{3}{121}x+\frac{567}{1131350})^3(x^2+\frac{41}{55}x-\frac{413}{102850}),
\end{split}
\end{equation}
where specializations of $s$ in the real interval $[-0.623.., -0.619..]$ (between the two algebraically conjugate branch points) lead to totally real fibers.
\par
Now all that is left is to parameterize the above extension $\cc(s)|\cc(t)$ over $\qq$ to fit the positions of the branch points. 
This leads to the following:
\begin{satz}
\label{pgl211a}
Let \[f_1(x):=x^3(x^2+x-\frac{413}{4114})^3(x^2-\frac{23}{726}x+\frac{63}{181016}),\]
\[f_2(x):=(x^2-\frac{3}{121}x+\frac{567}{1131350})^3(x^2+\frac{41}{55}x-\frac{413}{102850}),\] and
\[F(t,x):=f_1(x)^2 + \frac{27280791476537}{21954955473000}f_1(x)f_2(x) + \frac{766309482990625}{1985274409206528} f_2(x)^2 - tf_1(x)f_2(x) \in \qq(t)[x].\]
Then $F$ has regular Galois group $PGL_2(11)$ over $\qq(t)$ and possesses totally real specializations for all $t\mapsto t_0>135367$ (i.e.\ $t_0$ larger than the largest finite branch point).
The branch cycle structure with respect to $t$ is of type $(2^8.1^6, 2^{11}, 2^{11},3^6.1^4)$.
\end{satz}
\begin{proof}
$F$ is gained from the polynomial $f$ in \eqref{psl211_a} by setting 
\[t:=(s^2 + \frac{27280791476537}{21954955473000}s + \frac{766309482990625}{1985274409206528})/s.\]
We therefore first prove that $f$ has Galois group $PSL_2(11)$.\par
As in Section \ref{psl34}, we compute an explicit algebraic dependency for the natural (degree 54) cover of the reduced Hurwitz space over $\mathbb{P}^1$. 
We use this to find a second cover with the same ramification locus as the one given by $f$, and then make use of the fact that $PSL_2(11)$ has two non-conjugate subgroups of index 11. 
Set \[\tilde{s} = -(\frac{295}{726})^3 \cdot \frac{s^3\cdot (s^2+s+693/850)^3\cdot(s^2+1107/295\cdot s -5103/50150)}{(s^2+ 297/1475\cdot s -5103/1253750)^3\cdot(s^2+ 46/25\cdot s + 12474/10625)}.\]
Then $f(\tilde{s},x)$ splits over $\qq(s)$ into polynomials of degree 5 and 6. This shows that $Gal(f|\qq(s))$ has an intransitive index-11 subgroup, and so it cannot be equal to $A_{11}$ or $S_{11}$.
Dedekind reduction then leaves only $PSL_2(11)$.
So $f$ has Galois group $PSL_2(11)$ over $\qq(s)$, and regularity is obvious. Therefore $Gal(F|\qq(t))$ is a transitive subgroup of the wreath product $PSL_2(11)\wr C_2 < S_{22}$. 
Now one can check immediately that the only transitive subgroup of this wreath product with a generating 4-tuple (with product 1) of the necessary cycle structure is $PGL_2(11)$. 
So $PGL_2(11)$ is the geometric Galois group of $F$, and regularity follows because $PGL_2(11)$ is self-normalizing in $S_{22}$.\par
Finally, the assertion about totally real specializations is easy to verify.
\end{proof}
\subsection{The ramification type $(2A,2A,2B,4A)$}
\subsubsection{Hurwitz data and assumptions on branch points}
We consider another family, namely (in analogy to the above notation) the one associated to the class quadruple $(2A,2A,2B,4A)$ in $PGL_2(11)$.
Again, looking at the imprimitive action of $PGL_2(11)$ on 22 points, this monodromy leads to function fields $\cc(t) \subseteq \cc(s) \subseteq \cc(x)$. 
This time, the $PSL_2(11)$-part $\cc(x)|\cc(s)$ is ramified over 5 points, with monodromy of type $(2A,2A,2A,2A,2A)$. We therefore look for points on a reduced Hurwitz space of dimension 2. 
However, we do not need to parameterize the whole surface.\par
Suitable choice of the branch points in $\cc(t)$ and $\cc(s)$ leads to a model for a two-parameter polynomial, corresponding to a curve on the Hurwitz space.
Firstly, we can map the branch points of $\cc(t)$ to $0$, $\infty$ and $-1\pm\alpha$, with $\alpha^2\in \qq$ (for a rational model) and only the places at zero and infinity ramifying in $\cc(s)$. 
Therefore, by setting $t=s^2$, we may assume that the finite ramification locus of $s$ in $\cc(x)$ is $\pm \sqrt{-1-\alpha}$, $\pm \sqrt{-1+\alpha}$, 
and therefore the set of zeroes of the polynomial $s^4+2s^2+(1-\alpha^2)=:s^4+2s^2+\lambda$.
\par
We can use the braid criteria exhibited in \cite{De} to confirm the existence of a cover $\mathcal{C} \to \mathbb{P}^1$, where $\mathcal{C}$ is a curve of genus 1, 
parameterizing the polynomials with the above monodromy and restrictions on branch points. 
More precisely, our restrictions on the branch points lead to the same braids ($R_0:=\beta_1\beta_4$ and $R_1:=\beta_2\beta_3\beta_2$) as curve no.\ 14 on p.\ 49 in \cite{De}; 
the group generated by these braids acts intransitively on the inner Nielsen class of $PSL_2(11)$-generating tuples of type $(2A,2A,2A,2A,2A)$, 
with an isolated orbit of length 48 and corresponding braid orbit genus 1.
(Alternatively, observe that the 4-tuple in $PGL_2(11)$ with which we started to obtain the restrictions on the branch points has a Hurwitz curve of genus 1.)
\subsubsection{Computations}
As a starting point for the computations, we used a polynomial with 4 branch points and $PSL_2(11)$ monodromy, computed by Malle in \cite{Ma1}.
Develop this into a cover with 5 branch points (as done in the previous examples), and observe that the normalizer of an involution in $PSL_2(11)$ fixes one of the 2-cycles, 
therefore we can assume a polynomial equation $f(x)-s \cdot g_1(x)^2 \cdot g_2(x)=0$, with $deg(f)=11$ and $deg(g_i)=3$ 
(i.e. the infinite place of $\cc(x)$ lies over the infinite place of $\cc(s)$, with ramification index 2).
\par
Specializing the coefficients of $g_1$ and $g_2$ at $x^2$ to sufficiently many rational values again allowed an interpolation polynomial (of degree 4 in both variables),
and Magma computation again yields that this polynomial defines an elliptic curve of rank 1.\par
Now the procedure is the same as for the previous family: find a point on this curve that allows for a totally real fiber cover (one such point yields the polynomial
\begin{equation*}
\begin{gathered}
\label{psl211_b}
g(s,x):=(x-1)(x^5+9x^4+11x^3-65x^2-176x-1292/11)\\\cdot (x^5+14x^4-17/2x^3-18x^2+1/2x+5/11)
-s(x^3+x-14/11)^2(x^3+4x^2+5x+18/11)
\end{gathered}
\end{equation*}
with Galois group $PSL_2(11)$), and compose the resulting parameterization of $s$ as a rational function in $x$ with $t=s^2$.
\par
The Galois group can of course be verified just like in Theorem \ref{pgl211a}.
In this case, we also computed a degree-12 polynomial defining the stem field of a stabilizer in $PGL_2(11)$ in its natural action on 12 points. 
This is the polynomial $\tilde{g}$ in Theorem \ref{pgl211} below. It was found in the following way: Let $E$ be the splitting field of the above polynomial $g$ over $\qq(s)$. 
A primitive element of a subfield of $E$ of degree $12$ over $\qq(s)$ (corresponding to the stabilizer in $PSL_2(11)$ in its action on 12 points) can be computed with Magma. 
From this, one obtains a primitive element of the corresponding degree-12 extension of $\qq(s^2)$ as well.
By the Riemann-Hurwitz genus formula, this field is of genus 2. Therefore its gonality is 2. Via computation of Riemann-Roch spaces a rational subfield of index 2 can explicitly be parameterized. 
A few linear transformations then yielded the following polynomial:
\begin{satz}
\label{pgl211}
The polynomial 
{\footnotesize
\[\tilde{g}(t,x)=((x^3 + x^2 + \frac{1}{4}x + \frac{1}{22})^4 (t+1249)-364 (x^2 + \frac{5}{7}x - \frac{1}{44})\]
\[ (x^4 - \frac{137}{110}x^2 - \frac{3}{5}x - \frac{623}{9680}) (x^6 + \frac{36}{143}x^5 - \frac{323}{143}x^4 - \frac{6381}{3146}x^3 - \frac{9671}{25168}x^2 + \frac{5715}{138424}x - \frac{7035}{553696}))  t\]
\[-\frac{3^3\cdot 5^2\cdot 7\cdot 11}{4} (x^5 + 2x^4 + \frac{321}{550}x^3 - \frac{427}{550}x^2 - \frac{2771}{9680}x + \frac{401}{5324})^2 (x^2 + \frac{632}{693}x - \frac{6914}{22869})\]}
has regular Galois group $PGL_2(11)$ over $\qq(t)$.
The branch cycle structure with respect to $t$ is of type $(2^6, 2^6, 2^5.1^2, 4^3)$.
Furthermore, if $a$ is the unique branch point of $\tilde{g}$ inside $(0, +\infty)$, i.e.\ $a\approx 14.755$ the positive root of $t^2 + 1249t - 20511149/1100=0$, then
for $t_0\in (0,a)$, the specialized polynomial $\tilde{g}(t_0,x)$ is totally real.
\end{satz}
\section{Totally real extensions with group $G=PSL_3(3)$}
\label{totally_real2}
\subsection{A theoretical argument}
Computing totally real $PSL_3(3)$-extensions might be possible via covers with four branch points; however, there are no genus zero 4-tuples with a Hurwitz curve of genus zero in $PSL_3(3)$. 
We therefore solve the problem via a family of covers with five branch points, with branch cycle structure $(2^4.1^5,2^4.1^5,2^4.1^5,3^3.1^4,3^3.1^4)$ in the natural permutation representation
of $PSL_3(3)$. 
The reason is that this family can be seen to give rise to totally real $PSL_3(3)$-extensions via purely theoretical criteria:
\begin{prop}
\label{psl33orb}
In $PSL_3(3)$ (in its natural degree 13 action), let $2A$ be the class of involutions of cycle type $2^4.1^5$ and $3A$ be the class of elements of cycle type $3^3.1^4$.
Then the inner Hurwitz space of $C:=(2A,3A,2A,3A,2A)$ contains a rational genus zero curve over $\mathbb{Q}$, and therefore infinitely many $\mathbb{Q}$-points.
Furthermore, among these $\mathbb{Q}$ points, there are some that lead to totally real $PSL_3(3)$-polynomials.
\end{prop}
\begin{proof}
The group generated by the braids $B_0:=\beta_2\beta_3\beta_2$ and $B_1:=\beta_1^2\beta_4^2$ acts intransitively on the 120 $PSL_3(3)$-generating $5$-tuples of the straight inner Nielsen class 
$SNi^{in}((2A,3A,2A,3A,2A))$. This braiding action corresponds to curve no.\ (26) given on p.51 in Dettweiler's list of curves on Hurwitz spaces in \cite{De}. 
The orbits under this action are of lengths 12, 48 and 60; and the cycle structure of the braids in the action on the orbit of length 12 yields a (rational) genus zero curve on the Hurwitz space.
\par 
Alternatively, observe that the 4-tuple of classes in $Aut(PSL_3(3))$ (as an imprimitive permutation group on 26 points) with cycle structures $(2^8.1^{10}, 2^{13}, 3^6.1^8, 4^4.2^5)$ has braid orbit genus $g=0$. 
Our $PSL_3(3)$-5-tuple becomes a rational translate of this 4-tuple in a natural way, via ascending to the $PSL_3(3)$-fixed field. Therefore every rational point on the genus zero Hurwitz curve for the 4-tuple also 
yields a regular realization of $PSL_3(3)$ with the desired monodromy.
\par
The statement about totally real polynomials follows from group theoretic considerations. One only needs to find an element of our braid orbit where the identity element of $PSL_3(3)$
acts as complex conjugation on the branch cycles, as described in \cite[Thm.\ I.10.3]{MM}.
This yields the existence of $PSL_3(3)$-covers with totally real fibers, and as rational points are dense around real points on our $g=0$-Hurwitz curve, there are also such covers defined over
$\mathbb{Q}$.
\end{proof}
\subsection{Explicit computation}
As a starting point for the computations, we use a 4-point cover with group $PSL_3(3)$, with branch cycle structure $(2A,3A,3A,4A)$, as computed by Malle in \cite{Ma1}.
From this, the usual deformation process of Section \ref{deform}
leads to a 5-point cover with the above cycle structure, after writing the element of order $4$ as a product of two involutions in $PSL_3(3)$.\par
Once this is achieved, Proposition \ref{psl33orb} yields a recipe to compute a two-parameter family of
$PSL_3(3)$-polynomials - parametrized by a rational curve on the Hurwitz space - and specialize appropriately to obtain totally real extensions. 
This has been carried out in \cite[Chapter 8.2]{Koe}. The computations are analogous to the ones that have been performed several times by now. 
However, it turned out that the same ramification type also yields a three-parameter family of covers defined over $\mathbb{Q}$ - 
something that did not seem obvious from the theoretical arguments. We therefore describe the computation leading to this stronger result.
\subsection{A three-parameter family}
Explicit computations show that the reduced Hurwitz space $\mathcal{H}$, consisting of equivalence classes of covers with partially ordered branch point set 
$(\{\text{ zeroes of } t^3+t^2+at+b\},0,\infty)$ (with parameters $a,b$) and monodromy as above, does not only contain rational curves, but is in fact a rational surface.
Its function field is therefore of the form $\qq(\alpha,\beta)$ with independent transcendentals $\alpha$, $\beta$. 
In other words, there is a three-parameter family $f(\alpha,\beta,t,x)$ of $PSL_3(3)$-polynomials over $\qq(t)$, with branch point restrictions as above.
This family was found, beginning with any member of the two-parameter family in \cite[Chapter 8.2]{Koe}, by once again applying the techniques
of Section \ref{algdep}. Lifting an inital mod-$p$ solution to many different polynomials with the above restrictrion on branch points yielded an algebraic equation between three suitable
coefficients, say $\alpha, \gamma$ and $\delta$. Luckily, the curve given by this equation over the constant field $\qq(\alpha)$, was of genus zero, and even rational. 
Riemann-Roch space computations therefore yield its parameter $\beta$ - as a rational function in $\alpha, \gamma$ and $\delta$. Finally, algebraic dependencies between $\alpha$, $\beta$
and any of the remaining coefficients of the model lead to the following nice result:
\begin{satz}
\label{psl33}
The polynomial
\[f(\alpha,\beta,t,x):=f_0^3\cdot f_1\cdot x-t\cdot g_0^3\cdot g_1 \in \qq(\alpha,\beta)(t)[X], \text{ with }\]  
\[f_0:=x^3+\beta  x^2+(\beta-3)  x-\frac{1}{9} \alpha\beta^2 + \frac{4}{9} \alpha\beta - \frac{4}{3}\alpha,\]
\[f_1:=x^3+\frac{\alpha\beta^2 - 4\alpha\beta + 12\alpha - 3\beta^2 - 9}{(\beta-3)^2}  x^2+\frac{\alpha\beta^2 - 4 \alpha\beta + 12\alpha - 9\beta - 9}{3(\beta - 3)}  x-1,\]
\[g_0:=x^3+\alpha  x^2+\frac{1}{3} \alpha\beta  x+\frac{1}{9} \alpha\beta - \frac{1}{3} \alpha,\]
\[g_1:=\alpha  x^3+\frac{4 \alpha\beta - 3\alpha + 9}{3}  x^2+\frac{4 \alpha\beta^2 - 6 \alpha\beta + 9\alpha + 9\beta - 27}{9}  x-\alpha.\]
has regular Galois group $PSL_3(3)$ over $\qq(\alpha,\beta)$.
Suitable specializations for $\alpha,\beta$ and $t$ yield totally real $PSL_3(3)$-extensions. 
\end{satz}
\begin{proof}
The Galois group can again be verified using the two non-conjugate index-13 subgroups of $PSL_3(3)$.
As for totally real extensions, it would be somewhat complicated to classify all possibilities for specializations of $\alpha,\beta$ and $t$. We therefore content ourselves with
the special case $\alpha\mapsto-9$, $\beta\mapsto-6$. In this case, all choices $t\mapsto t_0$ with $t_0$ between the two smallest real branch points, i.e. $t_0\in  (-4.37.., -2.47..)$,
yield totally real specializations.
\end{proof}
The above family leads to $PSL_3(3)$-polynomials with various other ramification types as well. One particularly interesting observation is that
$f(\alpha,\beta,t,x)$ (as a polynomial in $x$) also defines a genus zero extension with respect to $\alpha$ (not just with respect to $t$!), 
although not in rational parameterization. The branch cycle structure with respect to $\alpha$ consists of six involutions (all of cycle structure $(2^4.1^5)$).
{\bf Remark:}\\
 The next open cases with regard to totally real Galois extensions occur for the permutation degree $n=14$: 
 there are no explicitly known totally real Galois extensions of $\qq$ with Galois group $PSL_2(13)$ or $PGL_2(13)$.
For these groups, the genus zero approach will no longer work. This is obvious for $PGL_2(13)$, as this group does not possess any generating genus zero tuples of length $\ge 4$. 
For $PSL_2(13)$, there is just one rational genus zero $4$-tuple (of cycle type $(2A,2A,2A,3A)$), with a Hurwitz curve of genus $g=1$.
One might therefore hope for an elliptic curve of rank $\ge 1$, as in the above $PGL_2(11)$-cases. However, explicit computation showed that this is an elliptic curve of rank zero 
(and more precisely, can be defined by $y^2=x^3 - 25x^2 + 136x - 180$), with no rational points leading to covers with real fibers. 
%
\section{Totally real extensions with groups $M_{22}$ and $Aut(M_{22})$}
\label{totally_real3}
The automorphism group $Aut(M_{22})=M_{22}.2$ of the Mathieu group $M_{22}$ has rational Hurwitz curves for genus zero $4$-tuples, which however do not give rise to totally real specializations.
We therefore computed the Hurwitz space for the family of cycles structures $(2^7.1^8, 2^8.1^6, 2^{11}, 6^2.3^2.2^2)$. The braid orbit is of length 30, and the Hurwitz curve of genus 1.
Once again, the elliptic curve turns out to be of rank one, and does indeed provide rational points belonging to totally real fibers. We only give one example:
\begin{satz}
\label{m22}
The polynomial 
$$f(t,x):=9180125(x^2 + 77x - 572)^6 (x^2 + 2816)^3 (2439x^2 - 10318x + 10912)^2 $$
$$- t  (367205x^8 + 59565800x^7 - 3770832472x^6 - 791515446176x^5 - 14589494734496x^4 + 556611262821376x^3 $$ $$+ 1682125644320768x^2 - 12791977299017728x + 14606802351030272)^2$$
$$  (405x^6 + 52290x^5 + 5828131/8   x^4 - 433357099/8  x^3 + 21649071627/32  x^2 - 3030076231x + 3867113368)$$
 has regular Galois group $Aut(M_{22})$ over $\qq(t)$, with ramification structure $(2^7.1^8, 2^8.1^6, 2^{11}, 6^2.3^2.2^2)$ with regard to $t$.
 For all $t_0 > 1$, the specialized polynomial $f(t_0,x)$ is totally real. \par
Furthermore, after setting $t:=t(s):=(11s^2+1)/(\frac{-13^2 \cdot 83 \cdot 194687^3}{2^3 \cdot 5^3 \cdot 11^{10} \cdot 271^2}  s^2 + 1))$, the polynomial
$g(s,x):=f(t(s),x)$ has regular Galois group $M_{22}$ over $\qq(s)$ 
and yields totally real specializations for all $s\mapsto s_0 \in \qq$ with $|s_0| < 0.0184$.
\end{satz}
\begin{proof}
The assertions about totally real specializations can be easily checked; also after proving the first assertion, one simply computes the discriminant to confirm that 
the Galois group of $g$ must be $Aut(M_{22}) \cap A_{22} = M_{22}$.\par
So we are left with showing that $Gal(f|\qq(t))\cong Aut(M_{22})$. By Dedekind reduction, one quickly sees that the only candidates are $Aut(M_{22})$ and $S_{22}$. 
To exclude the latter, one may use the fact that $Aut(M_{22})$ has an index-77 subgroup acting intransitively with orbits of degrees $6$ and $16$ 
(the stabilizer of a block of the $(3,6,22)$-Steiner system). As the six fixed points of the involution generating the inertia group at $t\mapsto \infty$ form such a block, expand the six simple poles
of $t=t(x)$ as series in $\frac{1}{t}$. The symmetric functions in these six elements then generate the fixed field $F$ of the Steiner block stabilizer. Sufficiently precise series yield an algebraic dependency
describing $F$, and as the Riemann-Hurwitz formula shows that this field is of genus 2, suitable Riemann-Roch space computations even yield an equation in two variables of degrees 2 and 3 for $F$.
Now express $t$ through these two variables; so far, everything has been based on approximations, but now verify that the polynomial $f(t,x)$ decomposes over (the genus 2 field) $F$ into factors of degree 6 and 16.
Riemann-Hurwitz shows that the fixed field of a 6-set stabilizer in $S_{22}$ would have much higher genus; this proves the assertion.
(Unfortunately, the occurring equations are too large to fit into this paper; however, note that once again this method of proof is rigorous and does not rely on numerical approximation as a monodromy computation would.)
\end{proof}
\section{Concluding remarks and applications}
All the Hurwitz spaces under consideration in the previous sections turned out to possess infinitely many rational points.
In particular, Theorems \ref{pgl211a}, \ref{pgl211} and \ref{m22} provide a single polynomial with real fibers, corresponding to a rational point on the respective Hurwitz curve. 
As mentioned above, there are actually infinitely many rational points, as the curves are elliptic of rank $rk>0$. 
As for any non-singular cubic curve $E$, defined over $\qq$, with $E(\qq)$ infinite, $\qq$-points lie dense 
(in the topology of $\mathbb{P}^2\rr$) around any given rational point, one obtains as an immediate corollary that the Hurwitz curves of all these families contain infinitely many points
with real fibers. This is because the property to possess real fibers is purely group theoretic and therefore locally invariant in the Hurwitz space.
\par
An obvious application of the parametric families is the search for number fields with prescribed Galois group (and possibly prescribed signature) and small discriminant, cf.\ the Kl\"uners-Malle 
database \cite{KM2}. We only give two examples. The first is a totally real $PGL_2(11)$-polynomial with root discriminant of small absolute value.
\begin{lemma}
The polynomial 
{\footnotesize\[g_0(x):=x^{12} - 4x^{11} - 220x^{10} - 88x^9 + 9768x^8 + 18480x^7 - 133760x^6 - 382272x^5\]\[ + 352880x^4 + 1664960x^3 + 455488x^2 - 994304x + 217152\]}
has Galois group $PGL_2(11)$ over $\qq$ and splitting field contained in $\rr$. The discriminant of a root field is equal to $2^{18} \cdot 3^5 \cdot 11^{13} \cdot 41^6 \approx 10^{31}$.
\end{lemma}
This polynomial is obtained from the polynomial $\tilde{g}$ in Theorem \ref{pgl211} by specializing $t\mapsto 27/10$ and then applying Magma's method \texttt{OptimizedRepresentation}.
\par
Also, the $PSL_3(3)$-family from Theorem \ref{psl33} has many specializations with ``small" discriminant in the sense that very few primes ramify. 
We conclude by giving a $PSL_3(3)$-number field ramified over one prime only.
\begin{lemma}
The polynomial {\footnotesize \[f_0(x):=x^{13} - 6x^{12} - 4542x^{11} - 226075x^{10} + 8156061x^9 + 770464590x^8\]
\[+ 11462215447x^7 - 970419905164x^6 - 33706100049495x^5 + 18705429494567x^4\]\[ + 29408002566579439x^3 + 237585722590314749x^2 - 
    2291157493210202812x - 11381632704121436976\]}
has Galois group $PSL_3(3)$, and only the prime $p=83420911386433$ ramifies in its splitting field.
In fact, a root field has discriminant $p^4$.
\end{lemma}
This polynomial is obtained from Theorem \ref{psl33}, via $\alpha:=148/69$, $\beta:=1$ and $t:=-1$, and again Magma's \texttt{OptimizedRepresentation}.
\section*{Acknowledgement}
I would like to thank Peter M\"uller for introducing me into many of the subjects of this work as well as carefully reading earlier versions,
J\"urgen Kl\"uners for informing me about some open cases of totally real Galois extensions,
and the anonymous referee for many helpful suggestions for improvements.

\end{document}